\documentclass[a4paper]{amsart}

\usepackage[utf8]{inputenc}
\usepackage[T1]{fontenc}
\usepackage{amsthm}
\usepackage{xspace}
\usepackage{amssymb,amsmath,stmaryrd}
\usepackage[english]{babel}
\usepackage[shortlabels]{enumitem}
\usepackage[all]{xy}
\usepackage{xcolor}
\usepackage{hyperref}
\usepackage{tikz-cd}
\usepackage{comment}

\theoremstyle{plain}
\newtheorem{theorem}{Theorem}[section]

\newtheorem{lemma}[theorem]{Lemma}

\theoremstyle{definition}

\newtheorem{definition}[theorem]{Definition}

\newtheorem{example}[theorem]{Example}
\newtheorem{question}[theorem]{Question}

\newtheorem*{question*}{Question}

\numberwithin{equation}{section}

\newcommand{\Z}{\ensuremath{\mathbb{Z}}\xspace}
\newcommand{\N}{\ensuremath{\mathbb{N}}\xspace}

\newcommand{\setof}[2]{\left\{ #1 \;\middle|\; #2 \right\}}

\title[The Algebraic Kirchberg-Phillips Question]{The Algebraic Kirchberg-Phillips Question for Leavitt path algebras}
\date{\today}
\author{Efren Ruiz}
\address{Department of Mathematics, University of Hawaii, Hilo, 200 W.~Kawili St., Hilo, Hawaii, 96720-4091 USA}
\email{ruize@hawaii.edu}

\thanks{The author is indebted to Gene Abrams and Mark Tomforde for their valuable comments during the process of writing this paper, and their insistence that this paper should be written.}

\begin{document}

%%%%%%%%%%%%%%%%%%%%%%%%%%%%%%%%%%%%%%%%%%%%%%%%%%%%%%%%%%
%%% SECTION: Abstract %%%%%%%%%%%%%%%%%%%%%%%%%%%%%%%%%%%%
%%%%%%%%%%%%%%%%%%%%%%%%%%%%%%%%%%%%%%%%%%%%%%%%%%%%%%%%%%

\begin{abstract}
The Algebraic Kirchberg-Phillips Question for Leavitt path algebras asks whether unital $K$-theory is a complete isomorphism invariant for unital, simple, purely infinite Leavitt path algebras over finite graphs. Most work on this problem has focused on determining whether (up to isomorphism) there is a unique unital, simple, Leavitt path algebra with trivial $K$-theory (often reformulated as the question of whether the Leavitt path algebras $L_2$ and $L_{2_-}$ are isomorphic).  However, it is unknown whether a positive answer to this special case implies a positive answer to the Algebraic Kirchberg-Phillips Question.  In this note, we pose a different question that asks whether two particular non-simple Leavitt path algebras $L_k(\mathbf{F}_*)$ and $L_k(\mathbf{F}_{**})$ are isomorphic, and we prove that a positive answer to this question implies a positive answer to the Algebraic Kirchberg-Phillips Question.  
\end{abstract}

\maketitle

\section{Introduction}

One of the most important open questions in the theory of Leavitt path algebras is the \emph{Algebraic Kirchberg-Phillips Question}, which asks whether the pointed $K_0$-group provides a complete isomorphism invariant for unital, simple, purely infinite Leavitt path algebras of finite graphs (see \cite[Question~5.4.8]{WCRH}).  This question first appeared over a decade ago in \cite{AALP08}, and at this point in time positive answers to the question are only known for special sub-classes of graphs and/or in the presence of additional conditions, such as requiring the matrices defining the $K_0$-groups to have determinants with the same sign \cite{AALP08, ALPS11}.  The success of most of these special cases relies heavily on Symbolic Dynamics results due to Franks \cite{Franks-flow} and Huang \cite{Huang-flow, Huang-auto}.  For one situation where isomorphism is achieved directly see \cite{AAP08}.  Most efforts after \cite{AALP08, AAP08, ALPS11} have focused on whether (up to isomorphism) there is a unique unital, simple, Leavitt path algebra with trivial $K$-theory.  The results of \cite{ALPS11} imply that there are at most two isomorphism classes, the classes defined by the Leavitt algebra $L_2$ and the Leavitt path algebra $L_{2_-}$.  The latter is the Leavitt path algebra of the graph obtained by Cuntz-splicing the graph defining $L_{2}$.  To determine whether these two classes are equal is known as the ``$L_{2}$--$L_{2_-}$-Question''.  

The $L_{2}$--$L_{2_-}$-Question and the Algebraic Kirchberg-Phillips Question were inspired by the corresponding questions for graph $C^*$-algebras (and Cuntz-Krieger algebras).  It was shown by Mikael R{\o}rdam in \cite[Theorem~7.2]{Ror95} (as outlined by Joachim Cuntz in \cite{Cun86}) that the uniqueness (up to $*$-isomorphism) of a unital, simple graph $C^*$-algebra with trivial $K$-theory is necessary and sufficient for $K$-theory to be a complete $*$-isomorphism invariant.  The key analytical tool used to obtain sufficiency is Voiculescu's theorem \cite{DV76} (see also \cite[Theorem~2]{MR1396721}), which allows one to extend an isomorphism between two unital, simple, purely infinite graph $C^*$-algebras with trivial $K$-theory to an isomorphism between two particular extension $C^*$-algebras.  The resulting $*$-isomorphism is then used to show the graph $C^*$-algebras associated to $E$ and its Cuntz-splice $E_-$ are Morita equivalent.  In \cite{ERRS-InvCS} we showed that Cuntz's argument could be adapted so that one could view these extension $C^*$-algebras as relative graph $C^*$-algebras, a realization we further exploited in \cite{ERRS-InvCS, segrerapws:gcgcfg, segrerapws:ccuggs} to complete the classification of all (possibly non-simple) unital graph $C^*$-algebras.

In the work of Cuntz and R{\o}rdam, the classification of unital, simple graph $C^*$-algebras is comprised of three main steps.

\medskip

\noindent \textbf{Step C1:} Prove that the $C^*$-algebras $\mathcal{O}_2$ and $\mathcal{O}_{2_-}$ are $*$-isomorphic (see \cite[Problem~2]{Cun86}).

\medskip

\noindent \textbf{Step C2:} Prove, using Voiculescu's theorem and the result from Step~C1, $C^*(E)$ and $C^*(E_{-})$ are Morita equivalent for any unital, simple, purely infinite graph $C^*$-algebra $C^*(E)$. 

\medskip

\noindent \textbf{Step C3:} Argue, using results from symbolic dynamics and the result from Step~C2, that any two unital, simple, purely infinite graph $C^*$-algebras with the same $K$-theory are $*$-isomorphic.

\medskip

Inspired by the classification of graph $C^*$-algebras, the authors of \cite{ALPS11} undertook a classification of Leavitt path algebras, formulating and making partial progress on the Algebraic Kirchberg-Phillips Question.  Further inspired by Cuntz and R{\o}rdam's work, the authors of \cite{ALPS11} outlined analogous steps that would provide a classification of  unital, simple, purely infinite, Leavitt path algebras over finite graphs.  Specifically, they described

\medskip

\noindent \textbf{Step L1:} Prove that $L_2$ and $L_{2_-}$ are isomorphic.

\medskip

\noindent \textbf{Step L2:} Embed $L_2$ and $L_{2_-}$ into $R := \operatorname{End} V$ for a certain vector space $V$, let $u$ be a particular element of $R$, let $\mathcal{E}_2$ be the subalgebra of $R$ generate by $L_2$ and $u$, and let $\mathcal{E}_{2_-}$ be the subalgebra of $R$ generate by $L_{2_-}$ and $u$.  Then prove that any isomorphism $\tau : L_2 \to L_{2_-}$ extends to an isomorphism $T : \mathcal{E}_2 \to \mathcal{E}_{2_-}$ such that $T(u) = u$.  (The precise statement, with complete definitions of all objects involved, can be found in \cite[p.223--224]{ALPS11}).

\medskip

\noindent \textbf{Step L3:} Argue, using results from symbolic dynamics and the result from Step~L2, that any two unital, simple, purely infinite Leavitt path algebras with the same $K$-theory are $*$-isomorphic.

\medskip

The main results of \cite{ALPS11} show that Step~L3 is true once Step~L1 and Step~L2 are established.  Step~L2 is significantly more involved than Step~C2, due to the fact that there is no algebraic analog of Voiculescu's theorem.  Consequently, Step~L2 describes (in purely algebraic terms) what is needed to account for this absence.

Since the publication of \cite{ALPS11} there has been intense focus on Step~L1, and the still open question of whether $L_2$ and $L_{2_-}$ are isomorphic.  By contrast, very little attention has been paid to Step~L2.  In fact, many newcomers to the subject of Leavitt path algebras, as well as some veterans who are not working directly on classification, have forgotten that answering the $L_2$--$L_{2_-}$-Question in the affirmative is only the first step in answering the Algebraic Kirchberg-Phillips Question.

The $L_2$--$L_{2_-}$-Question has been open for over a decade, suggesting that it may be a very difficult problem.  Furthermore, Step~L2 seems particularly daunting due to its abstract nature, the need to work in an ambient endomorphism ring containing copies of $L_2$ and $L_{2_-}$, and the requirement that one can extend any isomorphism of $L_2$ and $L_{2_-}$ in a particular way.

Inspired by our use of relative graph $C^*$-algebras that appears in applications of Voiculescu's theorem, we show in this note that we may replace Step~L1 and Step~L2 by a single step asking whether two particular non-simple Leavitt path algebras are isomorphic.  Specifically, we introduce the following question, where the graphs $\mathbf{F}_*$ and $\mathbf{F}_{**}$ are defined in Definition~\ref{def-fstar}.  These graphs are precisely the graphs used to identify certain relative Cohn path algebras as Leavitt path algebras (see paragraph after Definition~\ref{def-cohn-graph}).

\begin{question*}[The Cohn $L_{2}$--$L_{2_-}$-Question]
    Are the Leavitt path algebras $L_k(\mathbf{F}_*)$ and $L_k(\mathbf{F}_{**})$ isomorphic?
\end{question*}

We prove in Theorem~\ref{thm-main} that an affirmative answer to the Cohn $L_{2}$--$L_{2_-}$-Question implies a positive answer to the Algebraic Kirchberg-Phillips Question.  We shall also see that an affirmative answer to the Cohn $L_{2}$--$L_{2_-}$-Question implies an affirmative answer to the $L_2$--$L_{2_-}$-Question.

In light of these results, we argue that the $L_2$--$L_{2_-}$-Question and the Cohn $L_{2}$--$L_{2_-}$-Question should be used in tandem, providing a two-pronged approach as we seek an answer to the Algebraic Kirchberg-Phillips Question.  More specifically:  If one sets out to prove that the answer to the Algebraic Kirchberg-Phillips Question is YES, then one should try to establish an affirmative answer to the Cohn $L_{2}$--$L_{2_-}$-Question. Conversely, if one sets out to prove that the answer to the Algebraic Kirchberg-Phillips Question is NO, then one should try to establish a negative answer to the $L_2$--$L_{2_-}$-Question.

\section{Main Result}

Our main results will involve only  SPI graphs, which is the class of graphs whose Leavitt path algebras are precisely the simple, purely infinite Leavitt path algebras (see \cite[Theorems 2.9.7 and 3.1.10]{AAM-book}).

\begin{definition}
An \emph{SPI} graph is a directed graph that satisfies 
\begin{enumerate}[(1)]
\item every cycle has an exit;

\item every vertex connects to every singular vertex and every infinite path; and

\item the graph has at least one cycle.
\end{enumerate}
\end{definition}

The Algebraic Kirchberg-Phillips Question can now be phrased as follows.

\begin{question*}[The Algebraic Kirchberg-Phillips Question]\label{WCRH-question}
Let $\mathsf{k}$ be a field, and let $E$ and $F$ be two finite SPI graphs.  
Assume 
$$
(K_0( L_{\mathsf{k}}(E)) ,  [ 1_{L_{\mathsf{k}}(E)} ] ) \cong (K_0(L_{\mathsf{k}}(F)), [ 1_{L_{\mathsf{k}}(F)} ] ),
$$
that is, there exists an isomorphism $\varphi$ from $K_0( L_{\mathsf{k}}(E))$ to $K_0(L_{\mathsf{k}}(F))$ such that $\varphi( [1_{L_{\mathsf{k}}(E)} ] )= [ 1_{L_{\mathsf{k}}(F)} ]$ (such an isomorphism is called \emph{pointed}).  Is $L_{\mathsf{k}}(E)$ isomorphic to $L_{\mathsf{k}}(F)$ as rings?  
\end{question*}

For a graph $E$, let $A_E$ be the adjacency matrix of $E$.  Then the rows and columns of $A_E$ are indexed by the vertex set $E^0$ and the $(v,w)$-entry of $A_E$ is equal to the number of edges from $v$ to $w$.  By \cite[Corollary~2.7]{ALPS11}, for finite SPI graphs $E$ and $F$, if
$$
(K_0( L_{\mathsf{k}}(E)) , [ 1_{L_{\mathsf{k}}(E)} ],  \det( \mathsf{I} - A_E^t) ) \cong ( K_0( L_{\mathsf{k}}(F)) , [ 1_{L_{\mathsf{k}}(F)} ], \det( \mathsf{I} - A_F^t) ),
$$
that is, there is a pointed isomorphism between the $K_0$-groups and the indicated determinants are equal, then $L_{\mathsf{k}}(E)$ is isomorphic to $L_{\mathsf{k}}(F)$.  The determinant condition is the essential assumption to apply the results of Franks \cite{Franks-flow} and Huang \cite{Huang-flow, Huang-auto} to Leavitt path algebras.  For the graphs 
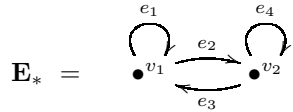
\begin{align*}
\mathbf{E}_* \  = \ \ \ \ \xymatrix{
  \bullet^{v_1} \ar@(ul,ur)[]^{e_{1}} \ar@/^/[r]^{e_{2}} & \bullet^{v_2} \ar@(ul,ur)[]^{e_{4}} \ar@/^/[l]^{e_{3}}
}
\end{align*}
\begin{align*}
\mathbf{E}_{**} \  =  \ \ \ \ 
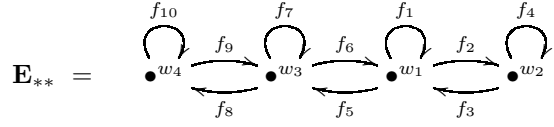
\xymatrix{
	\bullet^{ w_{4} } \ar@(ul,ur)[]^{f_{10}}  \ar@/^/[r]^{ f_{9} } & \bullet^{ w_{3} } \ar@(ul,ur)[]^{f_{7}} \ar@/^/[r]^{ f_{6} }  \ar@/^/[l]^{f_{8}} &  \bullet^{w_1} 				\ar@(ul,ur)[]^{f_{1}} \ar@/^/[r]^{f_{2}} \ar@/^/[l]^{f_{5}}
	& \bullet^{w_2} \ar@(ul,ur)[]^{f_{4}} \ar@/^/[l]^{f_{3}}
	}
\end{align*}
we have $K_0(L_{\mathsf{k}}(\mathbf{E}_*))$ and $K_0 (L_{\mathsf{k}}(\mathbf{E}_{**} ))$ are isomorphic to the trivial group with
$$
\det( \mathsf{I} - A_{\mathbf{E}_*}^t) =-1 \quad \text{and} \quad \det( \mathsf{I} - A_{\mathbf{E}_{**}}^t)=1.
$$ 
Thus, \cite[Corollary~2.7]{ALPS11} cannot be applied to prove that $L_{\mathsf{k}}(\mathbf{E}_*)$ and $L_{\mathsf{k}}(\mathbf{E}_{**})$ are isomorphic. The Leavitt path algebras of the graphs $\mathbf{E}_*$ and $\mathbf{E}_{**}$ are $L_2$ and $L_{2_-}$ respectively.  The graph $\mathbf{E}_{**}$ is what is called the Cuntz-Splice of $\mathbf{E}_*$ at the vertex $v_1$ (see Definition~\ref{def-cuntzsplice}).

To state our main result, we need the Cohn graphs of $\mathbf{E}_*$ and $\mathbf{E}_{**}$ which we will denote by $\mathbf{F}_*$ and $\mathbf{F}_{**}$ (see Definition~\ref{def-cohn-graph} and the paragraph after the definition).

\begin{definition}[The Cohn Graphs of $\mathbf{E}_*$ and $\mathbf{E}_{**}$]\label{def-fstar}
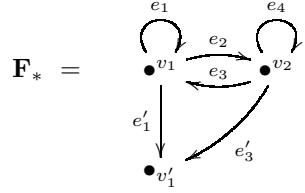
\begin{align*}
\mathbf{F}_* \ = \ \ \ \ \xymatrix{
  \bullet^{v_1} \ar[d]_{e_{1}'} \ar@(ul,ur)[]^{e_{1}} \ar@/^/[r]^{e_{2}} & \bullet^{v_2} \ar@(ul,ur)[]^{e_{4}} \ar@/^/[l]_{e_{3}} \ar@/^/[dl]^{e_{3}'} \\
  \bullet_{v_1'} & 
}
\end{align*} 
and
\begin{align*}
\mathbf{F}_{**} \ = \ \ \ \ 
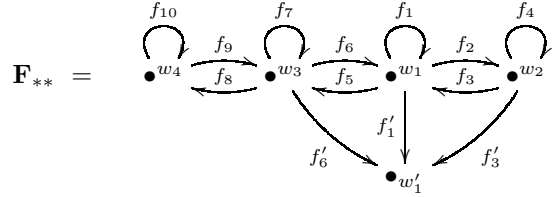
\xymatrix{
	\bullet^{w_4} \ar@(ul,ur)[]^{f_{10}}  \ar@/^/[r]^{ f_{9} } &
	\bullet^{w_3} \ar@/_/[dr]_{f_6'} \ar@(ul,ur)[]^{f_{7}} \ar@/^/[r]^{ f_{6} }  \ar@/^/[l]_{f_{8}} &
	\bullet^{w_1} \ar[d]_{f_1'} \ar@(ul,ur)[]^{f_{1}} \ar@/^/[r]^{f_{2}} \ar@/^/[l]_{f_{5}} &
	\bullet^{w_2} \ar@/^/[dl]^{f_3'} \ar@(ul,ur)[]^{f_{4}} \ar@/^/[l]_{f_{3}} \\
	& & \bullet_{w_1'} & 
}
\end{align*}
\end{definition}

The following theorem is the main result we shall prove in this paper.

\begin{theorem}\label{thm-main}
If the Leavitt path algebras $L_{\mathsf{k}}(\mathbf{F}_{*})$ and $L_{\mathsf{k}}(\mathbf{F}_{**})$ are isomorphic, then the Algebraic Kirchberg-Phillips Question has a positive answer. 
\end{theorem}

Theorem~\ref{thm-main} gives a sufficient condition for a positive answer to the Algebraic Kirchberg-Phillips Question.  Since the Leavitt path algebras $L_{\mathsf{k}}(\mathbf{F}_{*})$ and $L_{\mathsf{k}}(\mathbf{F}_{**})$  are non-simple, it is unknown whether or not the condition is necessary.

The Leavitt path algebras $L_{\mathsf{k}}(\mathbf{F}_*)$ and $L_{\mathsf{k}}( \mathbf{F}_{**})$ fit into the following short exact sequences of algebras 
\begin{center}
 $0 \to  \mathsf{M}_\infty(\mathsf{k}) \to L_{\mathsf{k}}(\mathbf{F}_*) \to L_2 \to 0 $ \\
 \medskip
 $0 \to \mathsf{M}_\infty(\mathsf{k}) \to L_{\mathsf{k}}(\mathbf{F}_{**}) \to L_{2_-} \to 0$
\end{center}
with $\mathsf{M}_\infty(\mathsf{k})$ the unique non-trivial ideal.  Thus, if an isomorphism between $L_2$ and $L_{2_-}$ lifts to an isomorphism between $L_{\mathsf{k}}(\mathbf{F}_*)$ and $L_{\mathsf{k}}(\mathbf{F}_{**})$, then the following questions are equivalent 
\begin{enumerate}
    \item The Algebraic Kirchberg-Phillips Question;
    \item The $L_2$--$L_{2_-}$-Question; and 
    \item The Cohn $L_2$--$L_{2_-}$-Question.
\end{enumerate}  

\begin{question*}[The $L_2$--$L_{2_-}$-Extension Question]\label{quet-necessary}
Does an isomorphism between $L_2$ and $L_{2_-}$ imply the existence of an isomorphism between $L_{\mathsf{k}}(\mathbf{F}_*)$ and $L_{\mathsf{k}}( \mathbf{F}_{**})$?
\end{question*}

To prove Theorem~\ref{thm-main}, we will need two graph constructions: the Cuntz-splice and the double Cuntz-splice.  These graph constructions are our main tools in order to reduce the proof of Theorem~\ref{thm-main} to a situation for which signs of determinants are the same, which then allows us to use \cite[Corollary~2.7]{ALPS11}.

\begin{definition}[The Cuntz-Splice]\label{def-cuntzsplice}
Let $E = (E^0 , E^1 , r , s )$ be a graph and let $u \in E^0$ be a regular vertex that supports at least two return paths.  Then $E_{u, -}$ is the graph defined as follows:
\begin{align*}
E_{u,-}^{0} &= E^{0} \sqcup \mathbf{E}_{*}^{0} \\
E_{u,-}^{1} &= E^{1} \sqcup \mathbf{E}_{*}^{1} \sqcup \{ d_1, d_2 \}
\end{align*}
with $r_{E_{u,-}} \vert_{E^{1}} = r_{E}$, $s_{E_{u,-}} \vert_{ E^{1} } = s_{E}$, $r_{E_{u,-}} \vert_{\mathbf{E}_{*}^{1}} = r_{\mathbf{E}_{*}}$, $s_{E_{u,-}} \vert_{\mathbf{E}_{*}^{1}} = s_{\mathbf{E}_{*}}$, and
\begin{align*}
	s_{E_{u,-}}(d_1) &= u	& r_{E_{u,-}}(d_1) &= v_{1} \\
	s_{E_{u,-}}(d_2) &= v_1	& r_{E_{u,-}}(d_2) &= u.
\end{align*}
We call ${E_{u,-}}$ the \emph{graph obtained by Cuntz splicing $E$ at $u$}.
\end{definition}

By \cite[Theorem~2.12]{ALPS11}, the Cuntz-splice is a graph construction that does not change the $K_0$-groups of the associated Leavitt path algebras (though it may change the position of the unit) but does change the sign of the determinant, that is, 
$$
K_0( L_{\mathsf{k}}(E)) \cong K_0( L_{\mathsf{k}}(E_{u,-})) \quad \text{and} \quad \det( \mathsf{I} - A_{E_{u,-}}^t) = -\det( \mathsf{I} - A_E^t).
$$
Note that the graph $\mathbf{F}_{**}$ is not isomorphic to the Cuntz-splice of $\mathbf{F}_*$ at $v_1$ as $\mathbf{F}_{**}$ has an edge from $w_3$ to the sink whereas the Cuntz-splice of $\mathbf{F}_*$ at $v_1$ does not.

\begin{definition}[The Double Cuntz-splice]
Let $E$ be a graph and let $u$ be a regular vertex that supports at least two return paths.  Then $E_{u,--}$ is the graph defined as follows:
\begin{align*}
E_{u,--}^{0} &= E^{0} \sqcup \mathbf{E}_{**}^{0} \\
E_{u,--}^{1} &= E^{1} \sqcup \mathbf{E}_{**}^{1} \sqcup \{ d_1, d_2 \}
\end{align*}
with $r_{E_{u,--}} \vert_{E^{1}} = r_{E}$, $s_{E_{u,--}} \vert_{ E^{1} } = s_{E}$, $r_{E_{u,--}} \vert_{\mathbf{E}_{**}^{1}} = r_{\mathbf{E}_{**}}$, $s_{E_{u,--}} \vert_{\mathbf{E}_{**}^{1}} = s_{\mathbf{E}_{**}}$, and
\begin{align*}
	s_{E_{u,--}}(d_1) &= u		& r_{E_{u,--}}(d_1) &= w_{1} \\
	s_{E_{u,--}}(d_2) &= w_1	& r_{E_{u,--}}(d_2) &= u.
\end{align*}
\end{definition}
Note that the graph $E_{u,--}$ is isomorphic to the Cuntz-splice of $E_{u,-}$ at $v_1$.  Consequently, by \cite[Theorem~2.12]{ALPS11},
$$
K_0( L_{\mathsf{k}}( E_{u,--} )) \cong K_0 ( L_{\mathsf{k}}( E )) \quad \text{and} \quad \det( \mathsf{I} - A_{E_{u,--}}^t) = \det( \mathsf{I} - A_{E}^t).
$$

\begin{example}\label{example:cuntz-splice-new}
Consider the graph 
\begin{align*}
E \  = \ \ \ \ \xymatrix{
  \bullet_{u} \ar@(ul,dl) \ar@(dr,ur) \ar@(ld,rd)
}
\end{align*}
\\

\noindent Then 
\begin{align*}
E_{u,--} \  = \ \ \ \ \xymatrix{
	\bullet^{ w_{4} } \ar@(ul,ur)[]^{f_{10}}  \ar@/^/[r]^{ f_{9} } &
	\bullet^{ w_{3} } \ar@(ul,ur)[]^{f_{7}} \ar@/^/[r]^{ f_{6} }  \ar@/^/[l]_{f_{8}} &
	\bullet^{w_1} \ar@/^/[d]^{d_2} \ar@(ul,ur)[]^{f_{1}} \ar@/^/[r]^{f_{2}} \ar@/^/[l]_{f_{5}} &
	\bullet^{w_2} \ar@(ul,ur)[]^{f_{4}} \ar@/^/[l]_{f_{3}} \\
	& & \bullet_{u} \ar@/^/[u]^{d_1} \ar@(ul,dl) \ar@(dr,ur)  \ar@(ld,rd)& \\
 & & &
}
\end{align*}
\end{example}

The proof of  Theorem~\ref{thm-main} proceeds as follows.  Assume  $L_{\mathsf{k}}(\mathbf{F}_{*}) \cong L_{\mathsf{k}}(\mathbf{F}_{**})$ and  assume $E$ and $F$ are finite SPI graphs with
$$
( K_0(L_{\mathsf{k}}(E)) , [1_{L_{\mathsf{k}}(E)}] ) \cong (K_0(L_{\mathsf{k}}(F)),  [1_{L_{\mathsf{k}}(F)}] ).
$$  
The crucial idea is to prove that $L_{\mathsf{k}}(E)$ and $L_{\mathsf{k}}(F)$ are Morita equivalent as \cite[Theorem~2.5]{ALPS11} allows us to conclude from this that $L_{\mathsf{k}}(E) \cong L_{\mathsf{k}}(F)$.  The proof that $L_{\mathsf{k}}(E)$ and $L_{\mathsf{k}}(F)$ are Morita equivalent requires two main steps.

\begin{enumerate}
    \item Use \cite[Theorem~1.25]{ALPS11} to prove the Leavitt path algebras $L_{\mathsf{k}}(E)$ and $L_{\mathsf{k}}(E_{u,--})$ are Morita equivalent and if $\det( \mathsf{I} - A_{E}^t) \neq \det( \mathsf{I} - A_{F}^t)$, then the Leavitt path algebras $L_{\mathsf{k}}(F)$ and $L_{\mathsf{k}}(E_{u,-})$ are Morita equivalent.

    \item Use any isomorphism between $L_{\mathsf{k}}(\mathbf{F}_{*})$ and $L_{\mathsf{k}}(\mathbf{F}_{**})$ to prove that for any vertex $u$, the Leavitt path algebras $L_\mathsf{k}( E_{u,-} )$ and $L_\mathsf{k}( E_{u,--} )$ are isomorphic.
\end{enumerate}

\section{Proof of the main result}\label{sec-proof-main}

Recall that two idempotents $e$ and $f$ in a ring $R$ are said to be \emph{Murray-von Neumann equivalent}, denoted by $e \sim f$, provided that there are $u, v \in R$ such that 
$$
e = uv \quad \text{and} \quad vu = f.
$$
Our first result shows that we can choose $u$ and $v$ that implements a Murray-von Neumann equivalence between two idempotents $e$ and $f$ such that $u \in eAf$ and $v \in fAe$.  This is well-known and appears as Exercise~2 in Chapter 7 of \cite{AF-book} but we have included the proof here for the convenience of the reader. 

\begin{lemma}\label{lem:murray-vNm-idempotents}
Let $A$ be a ring and let $e$ and $f$ be idempotents of $A$.  If $e \sim f$, then there are $x, y \in A$ such that 
\begin{align*}
xy&=e & x &=ex=xf  \\
 yx&=f & y&=fy=ye
\end{align*}
Consequently, $xyx=x$ and $yxy=y$.
\end{lemma}

\begin{proof}
Assume that $e\sim f$.  Then there are $v, w \in A$ such that $e=vw$ and $wv=f$.  Set $x= evf$ and $y=fwe$.  Then 
\begin{align*}
xy &= evfwe=e vwvwe=e^4=e \quad \text{and} \\
yx &= fwevf=fwvwvf = f^4=f. 
\end{align*}
A computation shows that $x = ex = xf$ and $y=fy=ye$.  It is now clear that $xyx=x$ and $yxy=y$.
\end{proof}

The next lemma is also well-known but since the proof is short  we have included the proof here for the reader's convenience.

\begin{lemma}\label{lem:invertible}
Let $A$ be a unital ring.  Suppose that $e_1, \ldots, e_n$ and $f_1, \ldots , f_n$ are two collections of mutually orthogonal idempotents of $A$ such that $e_i \sim f_i$ for all $i$ and 
$$\sum_{ i = 1}^n e_i = \sum_{i=1}^n f_i = 1.$$
Then there exists an invertible element $a \in A$ such that 
$$a^{-1} e_i a = f_i$$
for all $i$.
\end{lemma}

\begin{proof}
By Lemma~\ref{lem:murray-vNm-idempotents}, there are $x_i, y_i \in A$ such that $x_iy_i=e_i$, $y_ix_i=f_i$, $x_i=e_ix_i=x_i f_i$, $y_i=f_i y_i=y_ie_i$.  Set $a = \sum_{i=1}^n x_i$ and $b= \sum_{ i = 1}^n y_i$.  Then 
\begin{align*}
ab &= \sum_{ i,j=1}^n x_i y_j = \sum_{ i , j =1}^n x_i f_i f_j y_j = \sum_{ i=1}^n x_i y_i = \sum_{ i = 1}^n e_i = 1 \\
ba &= \sum_{ i,j=1}^n y_i x_j = \sum_{ i , j =1}^n y_i e_i e_j x_j = \sum_{ i=1}^n y_i x_i = \sum_{ i = 1}^n f_i = 1 \\
b e_k a &= \sum_{ i,j=1}^n y_i e_k x_j =  \sum_{ i,j=1}^n y_i e_i e_k e_j x_j = y_k x_k = f_k \quad \text{for all } k.\qedhere
\end{align*}
\end{proof}

It will be convenient to identify the Leavitt path algebras $L_{\mathsf{k}}( \mathbf{F}_{*})$ and $L_{\mathsf{k}}(\mathbf{F}_{**})$ as relative Cohn path algebras since there are fewer generators and relations in a relative Cohn path algebra (see \cite[Section~1.5]{AAM-book} and \cite{AK-Cohn, AM-simplelie}).  For a graph $E$, the set of \emph{regular vertices}, that is, the vertices of $E$ that emits at least one edge but does not emit infinitely many edges, will be denoted by $E_{\mathrm{reg}}^0$.

\begin{definition}[The relative Cohn path algebra at $V$]
Let $E =(E^0, E^1, r, s)$ be a graph, let $V \subseteq E_{\mathrm{reg}}^0$, and let $\mathsf{k}$ be a field.  The \emph{relative Cohn path $\mathsf{k}$-algebra at $V$}, $C_\mathsf{k}(E, V)$, is the universal $\mathsf{k}$-algebra generated by $\{ v, e, e^* : v \in E^0, e \in E^1 \}$ subject to the relations
\begin{enumerate}
    \item $vw = \delta_{v,w} v$ for all $v, w \in E^0$,

    \item $s(e) e = e r(e)=e$ for all $e \in E^1$, 
    
    \item $r(e) e^* = e^* s(e)=e^*$ for all $e \in E^1$,

    \item $e^*f = \delta_{e,f} r(e)$ for all $e, f \in E^1$, and 

    \item $v = \sum_{ s(e)=v } ee^*$ for all $v \in V$.
\end{enumerate}
\end{definition}

\begin{theorem}[{\cite[Theorem~1.5.18]{AAM-book}}]\label{thm:relative-graph}
Let $E =(E^0, E^1, r, s)$ be a graph and let $V \subseteq E_{\mathrm{reg}}^0$.  Define $E(V)$ to be the graph with vertices $E^0 \cup \{ v' : v \in E_\mathrm{reg}^0 \setminus V \}$ and edges $E^1 \cup \{ e' : e \in E^1 , r(e) \in  E_\mathrm{reg}^0 \setminus V \}$, and $r$ and $s$ extended to $E(V)^1$ by defining $s(e') = s(e)$ and $r(e') = r(e)'$.  Then for any field $\mathsf{k}$, 
$$
C_{\mathsf{k}}(E,V) \cong L_{\mathsf{k}}(E(V))$$
as $\mathsf{k}$-algebras.  Moreover, the isomorphism sends
\begin{align*}
    v &\mapsto \begin{cases}
        v+v' &\text{ if } v \in E_\mathrm{reg}^0\setminus V \\
        v &\text{ otherwise }
    \end{cases} \\
    e &\mapsto \begin{cases}
        e + e' &\text{ if } r(e) \in E_\mathrm{reg}^0 \setminus V \\
        e &\text{ otherwise}
    \end{cases} \\
    e^* &\mapsto \begin{cases}
        e^* + (e')^* &\text{ if } r(e) \in E_\mathrm{reg}^0 \setminus V \\
        e^* &\text{ otherwise}
    \end{cases} 
\end{align*}
\end{theorem}

\begin{definition}[The Cohn graph of $E$ at $V$]\label{def-cohn-graph}
    Let $E =(E^0, E^1, r, s)$ be a graph and let $V \subseteq E_\mathrm{reg}^0$.  We call the graph $E(V)$ in Theorem~\ref{thm:relative-graph} the \emph{Cohn graph of $E$ at $V$}.  
\end{definition}

Note that the Cohn graph of $\mathbf{E}_*$ at $\{v_2\}$ is
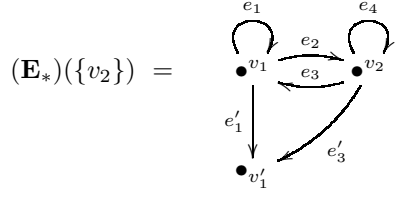
\begin{align*}
(\mathbf{E}_*)(\{v_2\}) \ = \ \ \ \ \xymatrix{
  \bullet^{v_1} \ar[d]_{e_{1}'} \ar@(ul,ur)[]^{e_{1}} \ar@/^/[r]^{e_{2}} & \bullet^{v_2} \ar@(ul,ur)[]^{e_{4}} \ar@/^/[l]_{e_{3}} \ar@/^/[dl]^{e_{3}'} \\
  \bullet_{v_1'} & 
}
\end{align*} 
and the Cohn graph of $\mathbf{E}_{**}$ at $\{w_2, w_3, w_4\}$ is
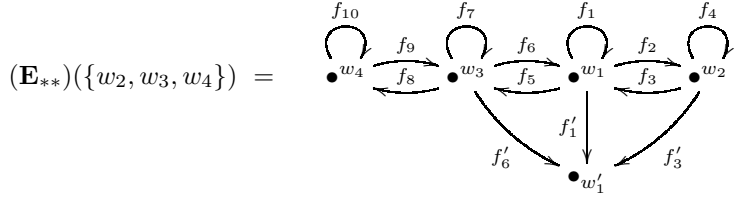
\begin{align*}
(\mathbf{E}_{**})(\{w_2,w_3,w_4\}) \ = \ \ \ \ \xymatrix{
	\bullet^{w_4} \ar@(ul,ur)[]^{f_{10}}  \ar@/^/[r]^{ f_{9} } &
	\bullet^{w_3} \ar@/_/[dr]_{f_6'} \ar@(ul,ur)[]^{f_{7}} \ar@/^/[r]^{ f_{6} }  \ar@/^/[l]_{f_{8}} &
	\bullet^{w_1} \ar[d]_{f_1'} \ar@(ul,ur)[]^{f_{1}} \ar@/^/[r]^{f_{2}} \ar@/^/[l]_{f_{5}} &
	\bullet^{w_2} \ar@/^/[dl]^{f_3'} \ar@(ul,ur)[]^{f_{4}} \ar@/^/[l]_{f_{3}} \\
	& & \bullet_{w_1'} & 
}
\end{align*}
which are the graphs $\mathbf{F}_*$ and $\mathbf{F}_{**}$ of Definition~\ref{def-fstar}.  Therefore, the assumption in Theorem~\ref{thm-main} is equivalent to an isomorphism between the relative Cohn path algebras $C_{\mathsf{k}}(\mathbf{E}_{*} , \{v_2\})$ and $C_{\mathsf{k}}(\mathbf{E}_{**}, \{w_2, w_3, w_4\} )$.

The following is the key lemma to build an isomorphism between $L_{\mathsf{k}}(E_{u,-})$ and $L_{\mathsf{k}}(E_{u,--})$ under the assumption that $C_{\mathsf{k}}(\mathbf{E}_{*}, \{v_2\})$ and $C_{\mathsf{k}}(\mathbf{E}_{**}, \{ w_2,w_3,w_4 \})$ are isomorphic.  

\begin{lemma}[{\cite[Lemma~4.4]{ERRS-InvCS}.}] \label{lem:csmurray-new}
Assume the relative Cohn path algebras $C_{\mathsf{k}}(\mathbf{E}_{*}, \{v_2\})$ and $C_{\mathsf{k}}(\mathbf{E}_{**}, \{ w_2,w_3,w_4 \})$ are isomorphic.  Set 
$$
\mathcal{E} = C_{\mathsf{k}}(\mathbf{E}_{*}, \{v_2\})$$ 
and choose an isomorphism between $C_{\mathsf{k}}(\mathbf{E}_{**}, \{ w_2,w_3,w_4 \})$ and $\mathcal{E}$.  Let $p_{w_1}$, $p_{w_2}$, $p_{w_3}$, $p_{w_4}$, $s_{f_1}$, $s_{f_2}, \ldots, s_{f_{10}}$, $s_{f_1^*}$, $s_{f_2^*}, \ldots, s_{f_{10}^*}$ denote the image of the canonical generators of $C_{\mathsf{k}}(\mathbf{E}_{**}, \{ w_2,w_3,w_4 \})$ in $\mathcal{E}$ under the chosen isomorphism. 
Then 
\begin{align*}
	e_1e_1^* + e_2 e_2^* &\ \sim \ s_{f_1} s_{f_1^*} + s_{f_2} s_{f_2^*} +s_{f_5} s_{f_5^*}, \\
	v_1 - \left( e_1 e_1^* + e_2 e_2^* \right) &\ \sim \ p_{w_1} - \left( s_{f_1} s_{f_1^*} + s_{f_2} s_{f_2^*} +s_{f_5} s_{f_5^*} \right), \\
	1_\mathcal{E}-v_1 = v_2 &\ \sim \ p_{w_2}+p_{w_3}+p_{w_4}=1_\mathcal{E}-p_{w_1}
\end{align*}
in $\mathcal{E}$.  Thus there exists an invertible element $z_0$ of $\mathcal{E}$ such that
\begin{align*}
	z_0\left(e_1 e_1^* + e_2 e_2^*\right)z_0^{-1} &= s_{f_1} s_{f_1^*} + s_{f_2} s_{f_2^*} +s_{f_5} s_{f_5^*}, \\
	z_0\left(v_1 - \left( e_1 e_1^* + e_2 e_2^* \right)\right)z_0^{-1} &= p_{w_1} - \left( s_{f_1} s_{f_1^*} + s_{f_2} s_{f_2^*} +s_{f_5} s_{f_5^*} \right), \\
	z_0 v_1z_0^{-1} &= p_{w_1} \\	
	 z_0 v_2 z_0^{-1} &= p_{w_2}+p_{w_3}+p_{w_4}.
\end{align*}
\end{lemma}

\begin{proof}
We first make the following observation of when two idempotents in $\mathcal{E}$ are Murray-von Neumann equivalent.  By \cite[Corollary~6.5]{MR2310414} and Theorem~\ref{thm:relative-graph}, $\mathcal{E}$ is a separative ring.  Suppose $p$ and $q$ are idempotents that generate the same two-sided ideal $I$ of $\mathcal{E}$ and $[p]=[q]$ in $K_0(I)$.  Then there are $m , n \in \N$ such that 
$$
p \oplus \left( \bigoplus_{k=1}^m q \right) \sim q \oplus \left( \bigoplus_{k=1}^m q \right), \ \bigoplus_{k=1}^m q \sim q' \leq \bigoplus_{k=1}^n q, \text{ and } \bigoplus_{k=1}^m q \sim q'' \leq \bigoplus_{k=1}^n p.  
$$
Thus, by \cite[Lemma~1.2]{AGOP-sep}, $p \sim q$.  Consequently, two idempotents in $\mathcal{E}$ are Murray-von~Neumann equivalent if and only if they generate the same ideal $I$ and have the same class in $K_0(I)$.  

By Theorem~\ref{thm:relative-graph}, we may realize the relative Cohn path algebras $C_{\mathsf{k}}(\mathbf{E}_{*}, \{v_2\})$ and $C_{\mathsf{k}}(\mathbf{E}_{**}, \{ w_2,w_3,w_4 \})$ as Leavitt path algebras of the graphs $(\mathbf{E}_*)(\{v_2\})$ and $(\mathbf{E}_{**})(\{w_2,w_3,w_4\})$ respectively.   Denote the image of the vertex idempotent of $L_{\mathsf{k}}((\mathbf{E}_*)(\{v_2\}))$ inside $\mathcal{E}$ under this isomorphism by $q_{v_1}, q_{v_2}, q_{v_1'}$ and denote the image of the vertex idempotents of $(\mathbf{E}_{**})(\{w_2,w_3,w_4\})$ inside $\mathcal{E}$ under the isomorphisms 
$$
L_{\mathsf{k}}((\mathbf{E}_{**})(\{w_2,w_3,w_4\}))\cong C_{\mathsf{k}}(\mathbf{E}_{**}, \{ w_2,w_3,w_4 \}) \cong\mathcal{E}
$$
by $q_{w_1}, q_{w_2}, q_{w_3}, q_{w_4}, q_{v_1'}$.  Using the description of the isomorphism in Theorem~\ref{thm:relative-graph}, we need to show that $q_{v_1} \sim q_{w_1}$, $q_{v_1'} \sim q_{w_1'}$ and $q_{v_2}\sim q_{w_2}+q_{w_3}+q_{w_4}$.

Since the smallest hereditary and saturated subset of $\mathbf{E}_*^0$ that contains $v_1$ is all of $(\mathbf{E}_*)(\{v_2\})^0$, we have that $q_{v_1}$ is a full idempotent in $\mathcal{E}$.  Similarly, $q_{w_1}$, $q_{v_2}$ and $q_{w_2}+q_{w_3}+q_{w_4}$ are full idempotents in $\mathcal{E}$. 
In $K_0(\mathcal{E})$, we have
\begin{align*}
[ q_{v_1} ] &= [ q_{v_1} ] + [ q_{v_2} ] + [ q_{v_1'} ] = [1_\mathcal{E} ]
& [ q_{v_2} ] &=  [ q_{v_1} ] + [ q_{v_2} ] + [ q_{v_1'} ] = [1_\mathcal{E} ]
\end{align*}
which imply that $[ q_{v_2} ] = -[ q_{v_1'}] =[ q_{v_1} ] $.  Therefore,
\[
[1_\mathcal{E} ] = [ q_{v_1} ] + [ q_{v_2} ] + [ q_{v_1'} ] = - [q_{v_1'} ] = [q_{v_1} ] = [q_{v_2}].
\]
In $K_0(\mathcal{E})$, we have
\begin{align*}
[ q_{w_1} ] &= [ q_{w_3} ] + [ q_{w_2} ] + [ q_{w_1} ] + [ q_{w_1'} ]  \\
[ q_{w_2} ] &= [ q_{w_2} ] + [ q_{w_1} ] + [ q_{w_1'} ]  \\
[ q_{w_3} ] &= [ q_{w_4} ] + [ q_{w_3} ] + [ q_{w_1} ] + [ q_{w_1'} ]  \\
[ q_{w_4} ] &= [ q_{w_4} ] + [ q_{w_3} ]
\end{align*}
which imply that $[ q_{w_3} ] = 0$ and $[q_{w_1} ] = - [ q_{w_1'} ]=[q_{w_2} ]$.  Hence,
\begin{align*}
[q_{w_4} ] &= - [ q_{w_1} ] - [q_{w_1'} ] = 0 \text{ and } \\
[ 1_{\mathcal{E}} ] &= [ q_{w_4} ] + [ q_{w_3} ] + [ q_{w_2} ] + [ q_{w_1} ] + [ q_{w_1'} ] = - [q_{w_1'}] = [q_{w_1} ] = [q_{w_2}] .
\end{align*}
Consequently,
\begin{align*}
	[q_{v_1}] &= [1_\mathcal{E}] = [q_{w_1}], \\
	[q_{v_2}] &= [1_\mathcal{E}] = [q_{w_2}]=[q_{w_2}]+[q_{w_3}]+[q_{w_4}].
\end{align*}
By the discussion in the first paragraph, $q_{v_1} \sim q_{w_1}$ and $q_{v_2}\sim q_{w_2}+q_{w_3}+q_{w_4}$.

Both $q_{v_1'}$ and $q_{w_1'}$ generate the only nontrivial ideal $I$ of $\mathcal{E}$.  Since the ideal is isomorphic to $\mathsf{M}_\infty(\mathsf{k})$ and both $[q_{v_1'}]$ and $[q_{w_1'}]$ are positive generators of  $K_0(I)\cong K_0(\mathsf{M}_\infty(\mathsf{k}))\cong \Z$, they must both represent the same class in $K_0(I)$. 
 Therefore, $q_{v_1'} \sim q_{w_1'}$.

The above paragraphs imply that \begin{align*}
	e_1e_1^* + e_2 e_2^* &\sim s_{f_1} s_{f_1^*} + s_{f_2} s_{f_2^*} +s_{f_5} s_{f_5^*}, \\
	v_1 - \left( e_1 e_1^* + e_2 e_2^* \right) &\sim p_{w_1} - \left( s_{f_1} s_{f_1^*} + s_{f_2} s_{f_2^*} +s_{f_5} s_{f_5^*} \right), \\
	1_\mathcal{E}-v_1 = v_2 &\sim p_{w_2}+p_{w_3}+p_{w_4}=1_\mathcal{E}-p_{w_1}.
\end{align*} 
The existence of $z_0$ now follows from Lemma~\ref{lem:invertible}.
\end{proof}

By Theorem~\ref{thm:relative-graph} and the remarks after the theorem, we have 
$$
L_{\mathsf{k}}(\mathbf{F}_{*}) \cong L_{\mathsf{k}}(\mathbf{F}_{**}).
$$
if and only if 
$$
C_{\mathsf{k}}(\mathbf{E}_{*}, \{v_2\}) \cong C_{\mathsf{k}}(\mathbf{E}_{**}, \{ w_2,w_3,w_4 \}).
$$
Thus, to prove Theorem~\ref{thm-main}, we may work with the above relative Cohn path algebras instead of the Leavitt path algebras $L_{\mathsf{k}}(\mathbf{F}_{*})$ and $L_{\mathsf{k}}(\mathbf{F}_{**})$.  

Our main tool in proving Theorem~\ref{thm-main} is the fact that the Leavitt path algebras over the Cuntz-splice and the double Cuntz-splice of a finite SPI graph are isomorphic when the relative Cohn path algebras 
\[
C_{\mathsf{k}}(\mathbf{E}_{*}, \{v_2\}) \quad \text{and} \quad C_{\mathsf{k}}(\mathbf{E}_{**}, \{ w_2,w_3,w_4 \})
\]
are isomorphic.  The proof of this fact follows the line of reasoning as the proof of \cite[Theorem~4.5]{ERRS-InvCS}.

\begin{theorem}\label{t:cuntz-splice-1}
Let $E$ be a finite SPI graph and let $u$ be a vertex of $E$ that supports at least two return paths.  If 
$$
C_{\mathsf{k}}(\mathbf{E}_{*}, \{v_2\}) \cong C_{\mathsf{k}}(\mathbf{E}_{**}, \{ w_2,w_3,w_4 \}),
$$ 
then $L_{\mathsf{k}}(E_{u,-}) \cong L_{\mathsf{k}}(E_{u,--})$.
\end{theorem}

\begin{proof}
As in Lemma~\ref{lem:csmurray-new}, $\mathcal{E}$ will be the relative Cohn path algebra $C_{\mathsf{k}}(\mathbf{E}_{*}, \{v_2\})$, and we choose an isomorphism between $\mathcal{E}$ and $C_{\mathsf{k}}(\mathbf{E}_{**}, \{ w_2,w_3,w_4 \})$, which exists by assumption.  Since $L_{\mathsf{k}}(E_{u,-}) \oplus \mathcal{E} $ is isomorphic to the Leavitt algebra of the finite graph $E_{u,-} \sqcup \mathbf{F}_*$, by \cite[Theorem~4.1]{embedd-l2}, there exists a unital, injective homomorphism
\[
	L_{\mathsf{k}}(E_{u,-}) \oplus \mathcal{E} \hookrightarrow L_2.
\]
We will suppress this embedding in our notation. 

In $L_2$, we denote the canonical generators of $L_{\mathsf{k}}(E_{u,-})$ by $p_v, v\in E_{u,-}^0$ and $s_e, s_{e^*},e\in E_{u,-}^1$, respectively, and we denote the canonical generators of $\mathcal{E}=C_{\mathsf{k}}(\mathbf{E}_*, \{v_2\})$ by $p_1,p_2$ and $s_1, s_2, s_3, s_4, s_1^*, s_2^*, s_3^*, s_4^*$, respectively.  Note that 
$$
\{ p_v : v \in E_{u,-}\} \cup \{ s_e, s_{e^*} : e \in E_{u,-}\}
$$
is a Cuntz-Krieger $E_{u,-}$-family in $L_2$.

We will define a second Cuntz-Krieger $E_{u,-}$-family in $L_2$. 
We let
\begin{align*}
q_v&=p_v&&\text{for each }v\in E^0, \\
q_{v_1}&=p_1, \\
q_{v_2}&=p_2.
\end{align*}
Since any two nonzero idempotents in $L_2$ are Murray-von~Neumann equivalent, by Lemma~\ref{lem:murray-vNm-idempotents}, we can choose $x_1, y_1, x_2, y_2 \in L_2$ such that 
\begin{align*}
	x_1 y_1 &= s_{d_1} s_{d_1^*}			& x_1 &= s_{d_1}s_{d_1^*} x_1 = x_1 p_1	 \\
	 y_1 x_1 &= p_1    & y_1 &= p_1 y_1 = y_1 s_{d_1} s_{d_1^*} \\
	x_2 y_2 &= p_1 - (s_1 s_1^* + s_2 s_2^*) & x_2 &= (p_1 - (s_1 s_1^* + s_2 s_2^*)) x_2= x_2 p_u \\
	 y_2 x_2 &= p_u & y_2 &= p_u y_2 = y_2 ( p_1 - (s_1 s_1^* + s_2 s_2^*) ).
\end{align*} 
We let 
\begin{align*}
t_e&=s_e&&\text{for each }e\in E^1, \\
t_{e^*}&=s_{e^*},&& \\
t_{e_i}&=s_i&&\text{for each }i=1,2,3,4, \\
t_{e_i^*}&=s_i^*,&& \\
t_{d_1}&=x_1, \\
t_{d_1^*}&=y_1, \\
t_{d_2}&=x_2, \\
t_{d_2^*}&=y_2. 
\end{align*}

By the choice of $\setof{ t_e, t_{e^*} }{ e \neq d_1,d_2 }$ the Cuntz-Krieger relations are clearly satisfied at all vertices other than $v_1$ and $u$. 
Additionally, the choices of $x_1, x_2$ ensure that the relations hold at $u$ and $v_1$ as well. 
Hence $\{q_v, t_e, t_{e^*} \}$ does indeed form a Cuntz-Krieger $E_{u,-}$-family in $L_2$.  Denote this family by $\mathcal{S}$.  Using the universal property of Leavitt path algebras, we get a homomorphism from $L_{\mathsf{k}}(E_{u,-})$ onto $A_\mathcal{S} \subseteq L_2$, where $A_\mathcal{S}$ is the subalgebra generated by $\mathcal{S}$.   Since $E$ is a SPI graph, $E_{u, -}$ is a SPI graph, and hence $L_{\mathsf{k}}(E_{u,-})$ is a simple ring.  Consequently, the homomorphism from $L_{\mathsf{k}}(E_{u,-})$ to $A_\mathcal{S}$ is in fact an isomorphism since it is a nonzero homomorphism onto $A_\mathcal{S}$.

Let $A_0$ be the subalgebra of $L_2$ generated by $\setof{ p_v }{ v \in E^0 }$, and let $A$ be the subalgebra of $L_2$ generated by $\setof{ p_v }{ v \in E^0 }$ and $\mathcal{E}$.  A computation shows $A=\mathcal{E} \oplus A_0$.  Let us denote by $\setof{ r_{w_i}, y_{f_j} , y_{f_j^*}}{ i=1,2,3,4, j = 1,2,\ldots, 10 }$ the image of the canonical generators of $C_{\mathsf{k}}(\mathbf{E}_{**}, \{ w_2,w_3,w_4 \})$ in $L_2$ under the chosen isomorphism between $C_{\mathsf{k}}(\mathbf{E}_{**}, \{ w_2,w_3,w_4 \})$ and $\mathcal{E}$ composed with the embedding into $L_2$. 

Let $z_0 \in \mathcal{E}$ be the invertible element in Lemma~\ref{lem:csmurray-new}, and set $z=z_0+\sum_{v\in E^0}p_v \in A$.  Clearly $z$ is an invertible element in $A$.  Since $A$ is a unital subalgebra of $A_\mathcal{S}$, $z$ is an invertible element in $A_\mathcal{S}$.  Hence, for all $x \in A_{\mathcal{S}}$, we have that $z x$ and $xz$ are elements of $A_ \mathcal{S}$.  By construction of $z$, we have that 
\begin{align*}
	z q_v  &= q_vz=q_v, \text{ for all } v \in E^0, \\
	z t_e  &= t_ez=t_e, \text{ for all } e \in E^1, \\
		z t_{e^*}  &= t_{e^*}z=t_{e^*}, \\
	z \left( t_{e_1} t_{e_1}^* + t_{e_2} t_{e_2}^* \right) z^{-1} &= y_{f_1} y_{f_1}^* + y_{f_2} y_{f_2}^* +y_{f_5} y_{f_5}^*, \\
	z \left( q_{v_1} - \left( t_{e_1} t_{e_1}^* + t_{e_2} t_{e_2}^* \right) \right) z^{-1}&= r_{w_1} - \left( y_{f_1} y_{f_1}^* + y_{f_2} y_{f_2}^* +y_{f_5} y_{f_5}^* \right), \\
z q_{v_1} z^{-1} &= r_{w_1}, \\ 
z q_{v_2} z^{-1} &= r_{w_2}+r_{w_3}+r_{w_4}.
\end{align*}

We will now define a Cuntz-Krieger $E_{u,--}$-family in $L_2$.
We let
\begin{align*}
P_v &= q_v = p_v &&\text{for each }v\in E^0, \\
P_{w_i}&=r_{w_i} &&\text{for each }i=1,2,3,4.
\end{align*}
Moreover, we let
\begin{align*}
S_e&=t_e=s_e &&\text{for each }e\in E^1, \\
S_{e^*}&=t_{e^*}=s_{e^*} &&\text{for each }e\in E^1, \\
S_{f_i}&=y_{f_i} &&\text{for each }i=1,2,\ldots,10, \\
S_{f_i^*}&=y_{f_i^*} &&\text{for each }i=1,2,\ldots,10, \\
S_{d_1}&=zt_{d_1}z^{-1}=zx_1z^{-1},\\
S_{d_1^*}&=zt_{d_1^*}z^{-1}=zy_1z^{-1},\\
S_{d_2}&=zt_{d_2}z^{-1}=zx_2z^{-1} ,\\
S_{d_2^*}&=zt_{d_2^*}z^{-1}=zy_2z^{-1}.
\end{align*}
Denote this family by $\mathcal{T}$.  By construction $\setof{ P_v }{ v \in E_{u,--}^0 }$ is a set of orthogonal idempotents, and 
\begin{align*}
S_e &=S_e p_{r(e)} = p_{s(e)} S_e,  & S_{e^*} &=S_{e^*} p_{s(e)} = p_{r(e)} S_{e^*}\\
S_{e^*}S_e&=s_{e^*}s_e=p_{r(e)}, 
& S_eS_{e^*}&=s_es_{e^*}, \\
S_{f_i^*}S_{f_i}&=y_{f_i^*}y_{f_i}=r_{f_i}, & 
S_{f_i}S_{f_i^*}&=y_{f_i}y_{f_i^*}, \\ 
S_{d_1^*}S_{d_1}&=r_{w_1}, & S_{d_1}S_{d_1^*}&=s_{d_1}s_{d_1^*}, \\ 
S_{d_2^*}S_{d_2}&=p_u, & S_{d_2}S_{d_2^*}&=r_{w_1}-\left(y_{f_1}y_{f_1^*}+y_{f_2}y_{f_2^*}+y_{f_5}y_{f_5^*}\right), 
\end{align*}
for all $e\in E^1$ and $i=1,2,\ldots,10$. From this, it is clear that $\mathcal{T}$ will satisfy the Cuntz-Krieger relations at all vertices in $E^0$.
Similarly, we see that since $\setof{ r_{w_i}, y_{f_j} , y_{f_j^*}}{ i=1,2,3,4, j = 1,2,\ldots, 10 }$ is a Cuntz-Krieger $(\mathbf{E}_{**}, \{ w_2,w_3,w_4 \})$-family, $\mathcal{T}$ will satisfy the relations at the vertices $w_2, w_3, w_4$. 
It only remains to check the summation relation at $w_1$, for that we compute
\begin{align*}
	\sum_{s_{E_{u,--}}(e) = w_1} S_e S_{e^*}	&= S_{f_1} S_{f_1^*} + S_{f_2} S_{f_2^*} + S_{f_5} S_{f_5^*} + S_{d_2} S_{d_2^*} \\
											&= y_{f_1} y_{f_1^*} + y_{f_2} y_{f_2^*} +y_{f_5} y_{f_5^*} + r_{w_1}-\left(y_{f_1}y_{f_1^*}+y_{f_2}y_{f_2^*}+y_{f_5}y_{f_5^*}\right) \\
											&= r_{w_1} = P_{w_1}.
\end{align*}
Hence $\mathcal{T}$ is a Cuntz-Krieger $E_{u,--}$-family. 

The universal property of $L_{\mathsf{k}}(E_{u,--})$ provides a surjective homomorphism from $L_{\mathsf{k}}(E_{u,--})$ to $A_{\mathcal{T}} \subseteq L_2$, where $A_\mathcal{T}$ is the subalgebra of $L_2$ generated by $\mathcal{T}$.   Since $E$ is SPI graph, $E_{u, --}$ is a SPI graph, and hence $L_{\mathsf{k}}(E_{u, --})$ is a simple ring.   Therefore, $L_{\mathsf{k}}(E_{u,--})$ is isomorphic to $A_\mathcal{T}$ as the homomorphism is clearly nonzero.

Recall that $z \in A$.  Therefore, $\mathcal{T} \subseteq A_\mathcal{S}$ since $A \subseteq A_\mathcal{S}$ and since $r_{w_i}, y_{f_j}, y_{f_j^*}\in \mathcal{E} \subseteq A_\mathcal{S}$, for $i=1,2,3,4$, $j = 1,2,\ldots, 10$.  So $A_\mathcal{T} \subseteq A_\mathcal{S}$.  Note that $p_v \in A_\mathcal{T}$ for all $v \in E^0$ and $r_{w_i} , y_{f_j}, y_{f_j^*} \in A_\mathcal{T}$ for all $i=1,2, 3, 4$ and $j = 1, 2,\ldots, 10$.  Therefore, $A = \mathcal{E} \oplus A_0 \subseteq A_\mathcal{T}$.  Therefore, $z = z_0+\sum_{v\in E^0}p_v \in A_\mathcal{T}$.  Consequently, $z x z^{-1}$ and $z^{-1} x z$ are elements of $A_\mathcal{T}$.  Hence, $\mathcal{S} \subseteq A_\mathcal{T}$, which implies that $A_\mathcal{S}  \subseteq A_\mathcal{T}$.  Consequently, $A_\mathcal{S} = A_\mathcal{T}$.
Therefore 
\[
	L_{\mathsf{k}}(E_{u,-}) \cong A_\mathcal{S} =A_\mathcal{T} \cong L_{\mathsf{k}}(E_{u,--}).\qedhere
\]
\end{proof}

We now have all the tools to prove our main result.

\begin{proof}[Proof of Theorem~\ref{thm-main}]
Assume 
$$
(K_0( L_{\mathsf{k}}(E)), [1_{L_{\mathsf{k}}(E)]}] ) \cong (K_0( L_{\mathsf{k}}(F)), [1_{L_{\mathsf{k}}(F)]}] ).
$$
If, $\det( \mathsf{I} - A_E^t) =  \det( \mathsf{I} - A_F^t)$, then the desired conclusion follows from \cite[Corollary~2.7]{ALPS11}.  Assume $\det( \mathsf{I} - A_E^t) \neq  \det( \mathsf{I} - A_F^t)$.  As $K_0( L_{\mathsf{k}}(E)) \cong K_0( L_{\mathsf{k}}(F))$, it is well-known that $\det( \mathsf{I} - A_F^t) = -\det( \mathsf{I} - A_E^t)$.  Therefore,
$$
K_0( L_{\mathsf{k}}( E_{u,--} )) \cong K_0 ( L_{\mathsf{k}}( E )) \quad \text{and} \quad \det( \mathsf{I} - A_{E_{u,--}}^t) = \det( \mathsf{I} - A_{E}^t).
$$
and 
$$
K_0( L_{\mathsf{k}}( E_{u,-} )) \cong K_0 ( L_{\mathsf{k}}( F )) \quad \text{and} \quad \det( \mathsf{I} - A_{E_{u,-}}^t) = \det( \mathsf{I} - A_{F}^t).
$$ 
By \cite[Theorem~1.25]{ALPS11}, $L_{\mathsf{k}}(E)$ and $L_{\mathsf{k}}(E_{u,--})$ are Morita equivalent, and $L_{\mathsf{k}}(F)$ and $L_{\mathsf{k}}(E_{u,-})$ are Morita equivalent.  Hence, by Theorem~\ref{t:cuntz-splice-1}, $L_{\mathsf{k}}(E)$ and $L_{\mathsf{k}}(F)$ are Morita equivalent.  By \cite[Theorem~2.5]{ALPS11}, $L_{\mathsf{k}}(E) \cong L_{\mathsf{k}}(F)$.  
\end{proof}

\section{Closing remarks}

Essentially all known positive answers to the Algebraic Kirchberg-Phillips Question have used results of Franks \cite{Franks-flow} and Huang \cite{Huang-flow, Huang-auto}, or been inspired by their techniques.  Thus, all results require that the sign of the appropriate determinants are equal. Consequently, we introduce the following question.   

\begin{question*}[The Sign of the Determinant Question]
    Let $E$ and $F$ be finite SPI graphs.  Suppose $L_\mathsf{k}(E)$ and $L_\mathsf{k}(F)$ are isomorphic.  Is $\mathrm{sgn}(\det( \mathsf{I}-A_E^t)) = \mathrm{sgn}(\det( \mathsf{I}-A_F^t))$? 
\end{question*}

Recall that $K_0(L_{\mathsf{k}}(\mathbf{E}_*))$ and $K_0 (L_{\mathsf{k}}(\mathbf{E}_{**} ))$ are isomorphic to the trivial group and
$$
\det( \mathsf{I} - A_{\mathbf{E}_*}^t) =-1 \quad \text{and} \quad \det( \mathsf{I} - A_{\mathbf{E}_{**}}^t)=1.
$$ 
Therefore, if the Sign of the Determinant Question has a positive answer, then the Algebraic Kirchberg-Phillips Question, and in particular the $L_2$--$L_{2_-}$-Question have negative answers.  If the $L_2$--$L_{2_-}$-Question has a positive answer, the Sign of the Determinant Question has a negative answer.

\subsection{The Algebraic Kirchberg-Phillips Question via Abrams et al.}  We now recall another hypothesis that leads to a positive answer for the Algebraic Kirchberg-Phillips Question given in \cite{ALPS11}.  We may represent $L_2$ and $L_{2_-}$ in $\mathrm{End}_{\mathsf{k}}(V)$, where $V$ is a $\mathsf{k}$-vector space of countable dimension with basis $\{v_i\}_{i=1}^\infty$.  Let $u \in \mathrm{End}_k(V)$ be the endomorphism defined by $u( v_i) = \delta_{1, i} v_1$, let $\mathcal{E}_2$ be the subalgebra of $\mathrm{End}_\mathsf{k}(V)$ generated by $L_2$ and $u$, and let $\mathcal{E}_{2_-}$ be the subalgebra of $\mathrm{End}_\mathsf{k}(V)$ generated by $L_{2_-}$ and $u$.  By \cite[Theorem~2.5 and Theorem~2.14]{ALPS11}, if there exists an isomorphism from $L_2$ to $L_{2_-}$ that extends to an isomorphism $T \colon \mathcal{E}_2 \to \mathcal{E}_{2_-}$ such that $T(u)=u$, then the Algebraic Kirchberg-Phillips Question has a positive answer.

Note that the approach presented in \cite{ALPS11} is a two-step process: (1) show $L_2 \cong L_{2_-}$ and (2) extend the isomorphism to an isomorphism between $\mathcal{E}_2$ and $\mathcal{E}_{2_-}$.  On the other hand the approach presented here is a one-step process.  All one needs to show is the existence of an isomorphism between the Leavitt path algebras $L_{\mathsf{k}}(\mathbf{F}_{*})$ and $L_{\mathsf{k}}(\mathbf{F}_{**})$.  It is unknown to the author whether or not there is a relationship between the two approaches.

\begin{question}
Are the statements 
\begin{enumerate}
    \item $L_{\mathsf{k}}(\mathbf{F}_{*}) \cong L_{\mathsf{k}}(\mathbf{F}_{**})$ and

    \item there exists an isomorphism from $L_2$ to $L_{2_-}$ that extends to an isomorphism $T \colon \mathcal{E}_2 \to \mathcal{E}_{2_-}$ such that $T(u)=u$
\end{enumerate}
equivalent?
\end{question}

\subsection{The analytic approach to classification of graph C*-algebras has no algebraic analogue}\label{subsec-analytic-class}

Much of the current efforts toward a resolution of the Algebraic Kirchberg-Phillips Question are toward determining whether $L_2$ and $L_{2_-}$ are isomorphic.  We show that the obvious algebraic translation of the approach taken by R{\o}rdam in \cite{Ror95} to prove that $\mathcal{O}_2$ and $\mathcal{O}_{2_-}$ are $*$-isomorphic fails to prove that $L_2$ and $L_{2_-}$ are isomorphic.  To show this, we remind the reader of the notion of approximately unitarily equivalent homomorphisms.

\begin{definition}\label{def-approx-unit-eq}
Let $A$ and $B$ be unital $C^*$-algebras.  Two $*$-homomorphisms $\varphi, \psi \colon A \to B$ are \emph{approximately unitarily equivalent} provided that for every $\epsilon >0$ and for every finite subset $\mathcal{F}$ of $A$, there exists a unitary $u$ in $B$ ($u^*u=uu^*=1$) such that 
$$
\| u \varphi(x) u^* - \psi(x) \| < \epsilon
$$
for all $x \in \mathcal{F}$.  

When $A = B$, we say that $\varphi$ is \emph{approximately inner} provided that $\varphi$ is approximately unitarily equivalent to $\mathrm{id}$.
\end{definition}

The key property used in the proof of R{\o}rdam is that any unital, $*$-endomorphism of $\mathcal{O}_2$ is approximately inner.  Similarly, any unital, $*$-endomorphism of $\mathcal{O}_{2_-}$ is approximately inner.  This is known in ``the $C^*$-algebra Classification Program'' as the \emph{uniqueness} part of the classification.  This, together with the existence of unital embeddings of $\mathcal{O}_2$ into $\mathcal{O}_{2_-}$ and $\mathcal{O}_{2_-}$ into $\mathcal{O}_2$, is used to show $\mathcal{O}_2$ and $\mathcal{O}_{2_-}$ are $*$-isomorphic through an \emph{Elliott Intertwining Argument}.  More precisely, we identify $\mathcal{O}_2$ and $\mathcal{O}_{2_-}$ as direct limits with the identity map for all the maps in the directed system, and we use uniqueness and existence theorems to build a diagram
\begin{equation*}
\begin{tikzcd}
    \mathcal{O}_2 \ar[rr, "\mathrm{id}"] \ar[rd]  & & \mathcal{O}_2  \ar[rd] \ar[rr, "\mathrm{id}"]  & &  \mathcal{O}_2 \ar[rd] \cdots   &  &\mathcal{O}_2   \arrow[d, dashrightarrow, bend right] \\
                        & \mathcal{O}_{2_-} \ar[ru] \ar[rr, "\mathrm{id}"] & &  \mathcal{O}_{2_-} \ar[ru] \ar[rr, "\mathrm{id}"]&  &  \cdots & \mathcal{O}_{2_-} \arrow[u, dashrightarrow, bend right]
\end{tikzcd}  
\end{equation*}
such that each triangle in the diagram approximately commutes on larger finite sets with an error tending to zero as one goes further to the right of the diagram.  These approximately commuting triangles are used to obtain an isomorphism between the direct limits, and hence an isomorphism between $\mathcal{O}_2$ and $\mathcal{O}_{2_-}$.

By \cite[Corollary~4.10]{embedd-l2}, there is unital embedding of $L_2$ into $L_{2_-}$, and there is a unital embedding of $L_{2_-}$ into $L_2$.  To mirror R{\o}rdam arguments, one would need to develop a ``uniqueness type'' theorem for unital endomorphisms of $L_2$ and unital endomorphisms of $L_{2_-}$.  As algebras may not have a metric nor unitaries, it is natural to define approximately unitarily equivalent homomorphisms of algebras as follows.

\begin{definition}\label{def-approx-inner}
Let $R$ and $S$ be unital rings.  Two homomorphisms $\varphi, \psi \colon R \to S$ are \emph{approximately unitarily equivalent} provided that for every finite subset $\mathcal{F}$ of $A$, there exists an invertible element $u$ in $S$ such that 
$$
 u \varphi(x) u^{-1} = \psi(x) 
$$
for all $x \in \mathcal{F}$.  

When $R = S$, we say that $\varphi$ is \emph{approximately inner} provided that $\varphi$ is approximately unitarily equivalent to $\mathrm{id}$.
\end{definition}

We show that there are unital endomorphisms of $L_2$ that are not approximately inner by showing first that any unital endomorphism of a Leavitt path algebra over a finite graph that is approximately inner is always an inner automorphism.  Then we show there are unital endomorphisms of $L_2$ that are not inner automorphisms.

\begin{theorem}\label{thm-approx-inner-is-inner}
    Let $E$ be a finite graph and let $\alpha \colon L_{\mathsf{k}}(E) \to L_{\mathsf{k}}(E)$ be a unital endomorphism that is approximately inner.  Then $\alpha$ is an inner automorphism.
\end{theorem}

\begin{proof}
    Since $\alpha$ is approximately inner, there exists an invertible element $u$ in $L_{\mathsf{k}}(E)$ such that $\alpha(x) = u x u^{-1}$ for all $x \in E^1 \cup (E^1)^* \cup E^0$ as $E^1 \cup (E^1)^* \cup E^0$ is a finite set by assumption.  Therefore, the homomorphisms $\alpha$ and $\mathrm{Ad}(u)$ agree on the generators of $L_{\mathsf{k}}(E)$ which implies $\alpha = \mathrm{Ad}(u)$.
\end{proof}

\begin{theorem}
There exists a unital, endomorphism of $L_2$ that is not approximately inner.
\end{theorem}

\begin{proof}
Define $\varphi \colon L_2 \to L_2$ by $\varphi(x) = (e_1+e_2) x (e_1+e_2)^* + (e_3+e_4) x(e_3+e_4)^*$.  Then $\varphi$ is a unital homomorphism.  We claim that $\varphi$ is not a surjection.  Suppose to contrary $\varphi$ is a surjection.  Then there exists $\ell \in L_2$ such that 
$$
(e_1+e_2) \ell (e_1+e_2)^* + (e_3+e_4) \ell(e_3+e_4)^* = (e_1+e_2)(e_3+e_4)^*.
$$
Consequently, 
\begin{align*}
1_{L_2}  &=(e_1+e_2)^*(e_1+e_2)(e_3+e_4)^*(e_3+e_4) \\
    &= (e_1+e_2)^*\left[ (e_1+e_2) \ell (e_1+e_2)^* + (e_3+e_4) \ell (e_3+e_4)^* \right](e_3+e_4) \\
    &= 0.
\end{align*}
Thus, $\varphi$ is not a surjection.  Since $\varphi$ is not a surjection, by Theorem~\ref{thm-approx-inner-is-inner}, $\varphi$ is not approximately inner. 
\end{proof}

\subsection{Equivalences of the various open questions in the classification program of Leavitt path algebras}  This section is devoted to reformulating the \emph{$L_2$--$L_{2_-}$-Question} and the \emph{Algebraic Kirchberg-Phillips Question}.  

\begin{question*}[The $L_2$--$L_{2_-}$-Question]
Are the Leavitt path algebras $L_2$ and $L_{2_-}$ isomorphic?
\end{question*}

\begin{question*}[The Algebraic Kirchberg-Phillips Question]
Let $\mathsf{k}$ be a field, and let $E$ and $F$ be two finite SPI graphs.  
Assume 
$$
(K_0( L_{\mathsf{k}}(E)) ,  [ 1_{L_{\mathsf{k}}(E)} ] ) \cong (K_0(L_{\mathsf{k}}(F)), [ 1_{L_{\mathsf{k}}(F)} ] ),
$$
that is, there exists an isomorphism $\varphi$ from $K_0( L_{\mathsf{k}}(E))$ to $K_0(L_{\mathsf{k}}(F))$ such that $\varphi( [1_{L_{\mathsf{k}}(E)} ] )= [ 1_{L_{\mathsf{k}}(F)} ]$.  Is $L_{\mathsf{k}}(E)$ isomorphic to $L_{\mathsf{k}}(F)$?
\end{question*}

We first show that the \emph{$L_2$--$L_{2_-}$-Question} is equivalent to determining whether there is (up to isomorphism) a unique  unital, simple, purely infinite Leavitt path algebra with trivial $K$-theory.

\begin{theorem}\label{thm-unique-trivial-kthy}
There exists (up to isomorphism) a unique unital, simple, purely infinite Leavitt path algebra with trivial $K$-theory if and only if $L_2$ and $L_{2_-}$ are isomorphic.
\end{theorem}

\begin{proof}
Let $E$ be a graph for which $L_{\mathsf{k}}(E)$ is a unital, simple, purely infinite Leavitt path algebra with trivial $K$-theory.  Note $E^0$ is finite since $L_\mathsf{k}(E)$ is unital.  Since $L_\mathsf{k}(E)$ is purely infinite, $E$ has no sinks.  Lastly, trivial $K$-theory implies that each vertex emits finitely many edges.  Hence, $E$ must be a finite graph with no sinks.  Since 
\[
\{ 0 \} = K_0(L_{\mathsf{k}}(E)) \cong \mathrm{coker}( \mathsf{I}-A_E^t),
\]
the Smith normal form of $\mathsf{I} - A_E^t$ contains only $-1$s and $1$s along its diagonal (see \cite{MN-book} for the definition of the Smith normal form of an integral matrix).  Thus,
$$
\det( \mathsf{I} - A_E^t) = 1 \quad \text{or} \quad \det( \mathsf{I} - A_E^t) = -1.
$$
Hence, by \cite[Corollary~2.7]{ALPS11}, $L_{\mathsf{k}}(E) \cong L_{2}$ or $L_{\mathsf{k}}(E) \cong L_{2,-}$.  The result now follow from this observation.
\end{proof}

The $L_2$--$L_{2_-}$-Question is a special case of the question of determining whether the Leavitt path algebras of $E$ and its Cuntz-splice are Morita equivalent.  To see this, note that $L_{2_-}$ is precisely the Leavitt path algebra of the Cuntz-splice of the graph defining $L_2$, and note that $L_2$ and $L_{2_-}$ are Morita equivalent if and only if $L_2 \cong L_{2_-}$ as $L_2$ and $L_{2_-}$ have trivial $K_0$-groups.  

\begin{question*}[The Invariance of the Cuntz-Splice Question]\label{qtn-inv-cuntzsplice}
Let $E$ be a finite SPI graph.  Is $L_{\mathsf{k}}(E)$ Morita equivalent to $L_{\mathsf{k}}(E_{u,-})$ for some $u \in E^0$?
\end{question*}

For $C^*$-algebras, Cuntz first presented this problem in \cite[Problem~1]{Cun86}.  It was resolved by R{\o}rdam \cite{Ror95}.

\begin{theorem}\label{thm-equivalent-CS-AKP}
The Invariance of the Cuntz-Splice Question has the same answer as the Algebraic Kirchberg-Phillips Question.
\end{theorem}

\begin{proof}
A computation shows that $K_0(L_{\mathsf{k}}(E)) \cong K_0( L_{\mathsf{k}}(E_{u,-}))$ but it is not necessarily true that there is a pointed isomorphism between their $K_0$-groups.  By adding a source to $E_{u,-}$ with one edge from the source to $u$, we get a graph $G$ such that $L_{\mathsf{k}}(G)$ and $L_{\mathsf{k}}(E_{u,-})$ are Morita equivalent, and 
$$
(K_0(L_{\mathsf{k}}(E)) , [ 1_{L_{\mathsf{k}}(E)} ] ) \cong (K_0(L_{\mathsf{k}}(G)) , [ 1_{L_{\mathsf{k}}(G)} ] ).
$$
Thus, a positive answer to the Algebraic Kirchberg-Phillips Question implies a positive answer to the Invariance of the Cuntz-Splice Question.  

Now, suppose there is a positive answer to the Invariance of the Cuntz-Splice Question, that is, for any finite SPI graph $E$, $L_{\mathsf{k}}(E)$ and $L_{\mathsf{k}}(E_{u,-})$ are Morita equivalent for some $u \in E^0$.  Now assume $( K_0(L_{\mathsf{k}}(E)), [1_{L_{\mathsf{k}}(E)} ] ) \cong ( K_0(L_{\mathsf{k}}(F)), [1_{L_{\mathsf{k}}(F)} ] )$ with $E$ and $F$ finite SPI graphs.  If, in addition, $\det( \mathsf{I} - A_E^t) = \det( \mathsf{I} - A_F^t)$, then by \cite[Corollary~2.7]{ALPS11}, $L_{\mathsf{k}}( E)$ is isomorphic to $L_{\mathsf{k}}(F)$.  Assume $\det( \mathsf{I} - A_E^t)\neq \det( \mathsf{I} - A_F^t)$.  Since $K_0(L_{\mathsf{k}}(E)) \cong K_0(L_{\mathsf{k}}(F))$, we have that $-\det( \mathsf{I} - A_E^t) = \det( \mathsf{I} - A_F^t)$.  Therefore, $K_0 ( L_{\mathsf{k}}(E_{u,-} ) ) \cong K_0(L_{\mathsf{k}}( E )) \cong K_0(L_{\mathsf{k}}(F))$ and $\det( \mathsf{I} - A_{E_{u,-}}^t ) = -\det( \mathsf{I} - A_E^t) = \det( \mathsf{I} - A_F^t)$.  By \cite[Theorem~1.25]{ALPS11}, $L_{\mathsf{k}}(E_{u,-})$ and $L_{\mathsf{k}}(F)$ are Morita equivalent which implies $L_{\mathsf{k}}(E)$ and $L_{\mathsf{k}}(F)$ are Morita equivalent as $L_{\mathsf{k}}(E)$ and $L_{\mathsf{k}}(E_{u,-})$ are assumed to be Morita equivalent.  Consequently, $L_{\mathsf{k}}(E)$ and $L_{\mathsf{k}}(F)$ are isomorphic by \cite[Theorem~2.5]{ALPS11}.
\end{proof}

%%%%%%%%%%%%%%%%%%%%%%%%%%%%%%%%%%%%%%%%%%%%%%%%%%%%%%%%%%
%%% SECTION: Bibliography %%%%%%%%%%%%%%%%%%%%%%%%%%%%%%%%
%%%%%%%%%%%%%%%%%%%%%%%%%%%%%%%%%%%%%%%%%%%%%%%%%%%%%%%%%%

\end{document}